\documentclass[11pt]{amsart}

\usepackage[margin=1in]{geometry}
\usepackage{amssymb}
\usepackage{hyperref}
\usepackage{mathrsfs}
\usepackage{amsmath,amscd}
\usepackage{color}

\newtheorem{theorem}{Theorem}[section]
\newtheorem{lemma}[theorem]{Lemma}
\newtheorem{proposition}[theorem]{Proposition}
\newtheorem{corollary}[theorem]{Corollary}

\theoremstyle{definition}

\theoremstyle{remark}
\newtheorem{remark}[theorem]{Remark}

\numberwithin{equation}{section}

\def\ZZ{\mathbb{Z}} \def\RR{\mathbb{R}}
 
\def\dgamma{\overline{\gamma}}
\def\ds{\rule{0pt}{1.5ex}}

\begin{document}

\title{Domination ratio of integer distance digraphs}
\author{Jia Huang}
\address{Department of Mathematics and Statistics, University of Nebraska at Kearney, Kearney, Nebraska, USA}
\curraddr{}
\email{huangj2@unk.edu}
\thanks{
}

\begin{abstract} 
An integer distance digraph is the Cayley graph $\Gamma(\ZZ,S)$ of the additive group $\ZZ$ of all integers with respect to some finite subset $S\subseteq\ZZ$.
The domination ratio of $\Gamma(\ZZ,S)$ is the minimum density of a dominating set in $\Gamma(\ZZ,S)$.
We establish some basic results on the domination ratio of $\Gamma(\ZZ,S)$ and precisely determine it when $S=\{s,t\}$ with $s$ dividing $t$.
\end{abstract}

\maketitle

\section{Introduction}\label{sec:intro}

Let $\Gamma=(V,E)$ be a \emph{digraph}, where $V$ is a set of objects called \emph{vertices} and $E\subseteq V\times V$ is a set of ordered pairs of vertices called \emph{(directed) edges}.
An edge $(u,v)\in E$ is from a vertex $u$ to another vertex $v$. 
The digraph $\Gamma$ is said to be \emph{finite} if $V$ and $E$ are both finite, or \emph{infinite} otherwise.
If every vertex of $\Gamma$ has only finitely many incoming and outgoing edges, then $\Gamma$ is said to be \emph{locally finite}.
If $(u,v)\in E\Longleftrightarrow (v,u)\in E$ for all $u,v\in V$, then we may view $\Gamma$ as an undirected graph by replacing each pair of opposite edges $(u,v)$ and $(v,u)$ with an undirected edge between $u$ and $v$.

Given vertices $u,v\in V$, we say $u$ \emph{dominates} $v$ if either $u=v$ or $(u,v)\in E$.
A set $D\subseteq V$ is called a \emph{dominating set} of the digraph $\Gamma$ if every vertex $v\in V$ is dominated by some vertex $u\in D$.
The concept of domination has wide applications in the real world, such as resource allocation.
It is a well-known NP-complete problem in graph theory to determine the \emph{domination number} $\gamma(\Gamma)$ of a finite digraph $\Gamma$, that is, the smallest cardinality of a dominating set of $\Gamma$.
The domination number has been extensively studied, and there are also many meaningful variations of domination, such as total domination.
The reader is referred to the monograph by Haynes, Hedetniemi, and Slater~\cite{FundDom} for an overview of this field.

We study domination in certain infinite Cayley graphs in this paper.
Let $G$ be a group and $S$ a subset of $G$.
The \emph{Cayley graph} $\Gamma(G,S) = (V,E)$ is a digraph with vertex set $V = G$ and edge set $E = \{(g,gs): g\in G,\ s\in S\}$.
When $S$ is closed under taking inverses, the digraph $\Gamma(G,S)$ can be viewed as an undirected graph since $(g,h)$ is an edge if and only if $(h,g)$ is an edge.

Recently there has been some research work on the existence of an efficient dominating set in a finite Cayley graph; see, e.g., Chelvam--Mutharasu~\cite{SubgroupEffDom} and Dejter--Serra~\cite{EffDomCayley}.
An \emph{efficient dominating set}, also called a \emph{perfect code}, of a digraph $\Gamma$ is a dominating set $D$ such that every vertex of $\Gamma$ is dominated by exactly one vertex in $D$.
For a finite Cayley graph $\Gamma=\Gamma(G,S)$ with $n=|G|$ vertices, each having $d=|S|$ outgoing edges, there is a straightforward lower bound $\gamma(\Gamma)\ge n/(1+d)$, where the equality holds if and only if there exists an efficient dominating set.

A \emph{circulant (di)graph} is a finite Cayley graph $\Gamma(\ZZ_n,S)$ where $\ZZ_n$ is the finite cyclic group of integers modulo $n$ and $S$ is a subset of $\ZZ_n$.
When $-S:=\{-s:s\in S\}$ coincides with $S$, the digraph $\Gamma(\ZZ,S)$ can be viewed as an undirected graph.
Circulant graphs provide important topological structures for interconnection networks due to their symmetry, fault-tolerance, routing capabilities, and other good properties, and have been used in telecommunication networks, VLSI design, and distributed computation.
Domination in circulant graphs has been studied by 
Huang--Xu~\cite{BonEffDomVertexTrans}, Kumar--MacGillivray~\cite{EffDomCirc}, Obradovi\'c--Peters--Ru\v{z}i\'c~\cite{EffDomCircChord}, Rad~\cite{DomCirc}, and others. 
Let $\gamma(\ZZ_n,S)$ denote the domination number of $\Gamma(\ZZ_n,S)$.
The following results are known for the \emph{double loop network} $\Gamma(\ZZ_n,\{1,s\})$ with $s\in\{2,3,\ldots,n-1\}$ or $\Gamma(\ZZ_n,\{\pm1,\pm s\})$ with $s\in\{2,3,\ldots,\lceil n/2 \rceil -1\}$. 

\begin{proposition}[{Huang and Xu~\cite{BonEffDomVertexTrans}}]\label{prop:eff}
{\normalfont(i)} 
If $1< s< \lceil n/2 \rceil$ then $\lceil n/5\rceil \le \gamma(\ZZ_n,\{\pm1,\pm s\})\le \lceil n/3\rceil$
and $\Gamma(\ZZ_n,\{\pm1,\pm s\})$ has an efficient dominating set if and only if $5\mid n$ and $s\equiv \pm2\pmod 5$.

{\normalfont(ii)} If $1< s< n$ then $\lceil n/3\rceil \le \gamma(\ZZ_n,\{1,s\}) \le \lceil n/2\rceil$ and $\Gamma(\ZZ_n,\{1,s\})$ has an efficient dominating set if and only if $3\mid n$ and $s\equiv2\pmod 3$.

{\normalfont(iii)} 
If $1\le s\le n-1$ then $\gamma(\ZZ_n,\{1,2,\ldots,s\}) = \lceil n/(s+1)\rceil$. 
\end{proposition}

\begin{proposition}[{Rad~\cite{DomCirc}}]
{\normalfont(i)} If $n\not\equiv 4\pmod 5$ then $\gamma(\ZZ_n,\{\pm1,\pm3\}) = \lceil n/5 \rceil$.

{\normalfont(ii)} If $n\equiv 4\pmod 5$ then $\gamma(\ZZ_n,\{\pm1,\pm3\})=\lceil n/5\rceil+1$.
\end{proposition}

Now we focus on domination in an \emph{integer distance (di)graph}, i.e., a Cayley graph $\Gamma(\ZZ,S)$ where $\ZZ$ is the infinite cyclic group of all integers under addition and $S\subseteq\ZZ$.
Our motivation is twofold.
On the one hand, integer distance graphs are natural generalizations of circulant graphs.
In fact, the chromatic number of an undirected integer distance graph $\Gamma(\ZZ,\pm S)$, where $\pm S:=\{\pm s:s\in S\}$, has been extensively studied before; see, e.g., Carraher, Galvin, Hartke, Radcliffe, and Stolee~\cite{IndepRatio}. 
On the other hand, an integer distance graph can be viewed as the limit of a sequence of circulant graphs and understanding domination in integer distance graphs may shed light on the asymptotic behavior of domination in large circulant graphs.

We assume $0\notin S$ throughout this paper, since removing an edge from a vertex $v$ to itself (i.e., a loop at $v$) has no effect on domination.
When $S$ is finite, the Cayley graph $\Gamma(\ZZ,S)$ is locally finite and a dominating set of $\Gamma(\ZZ,S)$ must be infinite, since every vertex dominates at most $|S|$ many other vertices.
To measure how large a possibly infinite subset $U$ of $\ZZ$ is, we define the \emph{(lower) density} of $U$ in $\ZZ$ as the following limit inferior 
\begin{equation}\label{eq:density}
\delta(U) := \liminf_{n\to\infty}  \frac{|U\cap [-n,n]|}{2n+1}. 
\end{equation}
For example, we have $\delta(U)=0$ and $\delta(\ZZ\setminus U)=1$ when $U$ is finite, and $\delta(U)=1/2$ when $U=2\ZZ$.
In general, one has $0\le \delta(U)\le 1$ for any $U\subseteq\ZZ$.
We define the \emph{domination ratio} $\dgamma(\ZZ,S)$ of the graph $\Gamma(\ZZ,S)$ to be the infimum of $\delta(D)$ over all dominating sets $D$ of $\Gamma(\ZZ,S)$.

Replacing limit inferior with limit superior in \eqref{eq:density} gives the \emph{upper density} of $U\subseteq\ZZ$.
Carraher, Galvin, Hartke, Radcliffe, and Stolee~\cite{IndepRatio} used upper density to study independent sets in an undirected integer distance graph $\Gamma(\ZZ,\pm S)$.
We provide some results on lower density in Section~\ref{sec:density}, with similar proofs to previous work~\cite{IndepRatio}.
For example, the following result proved in Section~\ref{sec:density} is a natural extension of an analogous result on independence ratio~\cite[Theorem~4]{IndepRatio}.

\begin{proposition}\label{prop:periodic}
Assume $S$ is a finite subset of $\ZZ\setminus\{0\}$.
Let $a$ and $b$ be the largest nonnegative integers in $S\cup\{0\}$ and $-S\cup\{0\}$, respectively.
Let $c:=a+b$.
Then the domination ratio of $\Gamma(\ZZ,S)$ is achieved by some periodic dominating set $D$ with period $p\le c2^c$.
\end{proposition}

Here a set $U\subseteq \ZZ$ is \emph{periodic} if there exists a positive integer $d$ such that 
\[ U\cap[id+1,id+d] = \{id+j: j\in U\cap[1,d]\}, \quad \forall i\in\ZZ.\]
The smallest such integer $d$ is called the \emph{period} of $U$.

The above result not only shows that the domination ratio of $\Gamma(\ZZ,S)$ is achieved by some periodic dominating set $D$, but also implies that the circulant digraph $\Gamma(\ZZ_p,S_p)$ has a minimum dominating set $D\cap[1,p]$ and its domination number is given by $\gamma(\ZZ_p,S_p) = |D\cap[1,p]| = \dgamma(\ZZ,S) p$, where $\ZZ_p:=\{1,2,\ldots,p\}$ is the cyclic group of order $p$ under addition modulo $p$ and $S_p$ is the subset of $\ZZ_p$ consisting of all the least positive residues of elements in $S$ modulo $p$. 
See Proposition~\ref{prop:finite}.

Next, we establish the following basic results on the domination ratio of $\Gamma(\ZZ,S)$ in Section~\ref{sec:basic}.

\begin{proposition}\label{prop:basic}
{\normalfont(i)} If $S\subseteq S'\subseteq \ZZ\setminus\{0\}$ then $\dgamma(\ZZ,S') \le \dgamma(\ZZ,S) = \dgamma(\ZZ,-S)$.

{\normalfont(ii)} If $|S|\le1$ then $\dgamma(\ZZ,S)=1/(1+|S|)$. If $2\le |S|<\infty$ then $1/(|S|+1)\le \dgamma(\ZZ,S)\le 1/2$.

{\normalfont(iii)} If $S$ is finite and there exists an efficient dominating set of $\Gamma(\ZZ,S)$, then $\dgamma(\ZZ,S)=1/(|S|+1)$.

{\normalfont(iv)} If $S = \{ i_1(s+1)+1,i_2(s+1)+2,\ldots,i_s(s+1)+s\}$ with $i_1,\ldots,i_s\in\ZZ$, then $\dgamma(\ZZ,S)=1/(s+1)$.

{\normalfont(v)} If $d$ divides all elements of $S$, then $\dgamma(\ZZ,S/d) = \dgamma(\ZZ,S)$, where $S/d:=\{s/d:d\in S\}$.
\end{proposition}

To further study the domination ratio $\dgamma(\ZZ,S)$, we observe that a dominating set of $\Gamma(\ZZ,S)$ can be written as $D=\{x_i:i\in\ZZ\}$, where $x_i<x_{i+1}$ for all $i\in \ZZ$, and it decomposes $\ZZ$ into a disjoint union of \emph{blocks} $B_i=\{x_i,x_i+1,\ldots,x_{i+1}-1\}$ for all $i\in\ZZ$. 
We develop some lemmas about blocks in Section~\ref{sec:block} and use them to prove the following result in Section~\ref{sec:cluster}.

\begin{theorem}\label{thm:1s}
{\normalfont(i)} For any integer $k$ we have $\dgamma(\ZZ,\{1,3k+2\}) = 1/3$.

{\normalfont(ii)} For any positive integer $k$ we have $\dgamma(\ZZ,\{1,3k+1\}) = \dgamma(\ZZ,\{1,-3k\}) = (k+1)/(3k+2)$.

{\normalfont(iii)} For any positive integer $k$ we have $\dgamma(\ZZ,\{1,3k\}) = \dgamma(\ZZ,\{1,-3k+1\}) = 2k/(6k-1)$.
\end{theorem}

Our proof of the above theorem uses certain partitions of the collection of all blocks obtained from a given dominating set.
This method is different from the one used in earlier work~\cite{IndepRatio} on the independence ratio of integer distance graphs.

Combining Proposition~\ref{prop:basic} (v) and Theorem~\ref{thm:1s} one can determine the domination ratio of $\Gamma(\ZZ,S)$ whenever $S$ consists of two distinct nonzero integers $s$ and $t$ with $s\mid t$.
If $t/s\equiv 2\pmod 3$ then $\dgamma(\ZZ,S)=1/3$; in this case $\Gamma(\ZZ,S)$ has an efficient dominating set by the proof of Theorem~\ref{thm:1s} (i).
If $t/s\not\equiv2\pmod 3$ then $\dgamma(\ZZ,S)$ is given by Theorem~\ref{thm:1s} (ii) and (iii), and since it is strictly larger than $1/3$, there exists no efficient dominating set for $\Gamma(\ZZ,S)$ in this case.

Combining Theorem~\ref{thm:1s} with Proposition~\ref{prop:finite} gives the domination number of certain circulant digraphs.
We have $\gamma(\ZZ_{3k+2},\{\pm1\}) = \gamma(\ZZ_{3k+2},\{1,2\})=k+1$, which agrees with Proposition~\ref{prop:eff} (ii), and $\gamma(\ZZ_{6k-1},\{1,3k\}) =2k$, for any positive integer $k$.
See Corollary~\ref{cor:finite}.

The existence of an efficient dominating set of $\Gamma(\ZZ,S)$ is equivalent to the ability to tile the integers with translates of $S\cup\{0\}$ (overlaps not allowed).
Researchers have extensively studied when a given set $X\subseteq\ZZ$ can tile the integers.
For example, Newman~\cite{Newman} solved this problem when the cardinality of $X$ is a power of a prime, and Coven and Meyerowitz~\cite{TilingIntegers} extended this to the case of at most two prime factors in the cardinality of $X$.
The result of Newman~\cite[Theorem~1, 2]{Newman} implies our characterization of the existence of an efficient dominating set of $\Gamma(\ZZ,\{s,t\})$ with $s\mid t$ as well as Proposition~\ref{prop:basic} (iv).
A lemma used by Coven and Meyerowitz~\cite[Lemma~1.2]{TilingIntegers}, first due to Haj\'{o}s~\cite{Hajos} and de Bruijn~\cite{deBruijn}, implies the special case of  Proposition~\ref{prop:periodic} when an efficient dominating set exists.
When there is no efficient dominating set, the investigation of the domination ratio of $\Gamma(\ZZ,S)$ would be natural and meaningful, as it tells us the most efficient ways to cover the integers with translates of $S\cup\{0\}$ (overlaps allowed).

\section{Density of a set of integers}\label{sec:density}

First recall that the \emph{limit inferior} of a sequence $(x_n)$ of real numbers is defined as
\[ \liminf_{n\to\infty} x_n := \lim_{n\to\infty} \Big(\inf_{m\ge n} x_m \Big) = \sup_{n\ge0} \Big(\inf_{m\ge n} x_m \Big). \]
This is either a real number or $\pm\infty$.
Moreover, it equals the ordinary limit of $(x_n)$ whenever the latter exists.
If $x_n\ge y_n$ for all sufficiently large $n$, then 
\[ \liminf_{n\to\infty} x_n \ge \liminf_{n\to\infty} y_n.\]

Next we generalize the density $\delta(U)$ of a subset $U\subseteq \ZZ$ to a weighted version.
Let $f:\ZZ\to\RR$ be a function.
For each nonempty finite set $A\subseteq\ZZ$ we write 
\[ \|f(A)\| := \sum_{a\in A} f(a).\] 
Define the \emph{density} of $f$ to be
\[ \delta(f):=\liminf_{n\to\infty} \frac{ \|f(\ZZ\cap[-n,n])\| }{2n+1}.\]
In particular, for any $U\subseteq \ZZ$, let $f=\chi_{\ds U}:\ZZ\to\RR$ be defined by 
\[ \chi_{\ds U}(i):=\begin{cases}
1, & i\in U,\\
0, & i\in\ZZ\setminus U.
\end{cases}\]
We define the \emph{density of $U$ in $\ZZ$} to be
\[ \delta(U) := \delta(\chi_{\ds U}) = \liminf_{n\to\infty}  \frac{|U\cap [-n,n]|}{2n+1}. \]
This agrees with the earlier definition~\eqref{eq:density} for $\delta(U)$.

The next two lemmas are extensions of some results in previous work on independence ratio~\cite[Lemma 17, 18]{IndepRatio}.
Lemma~\ref{lem:density} shows that the density can be calculated not only by looking over the interval $[-n,n]$, but also by looking at multiples of the interval and by making small bounded changes at both ends of the interval.

\begin{lemma}\label{lem:density}
Fix two positive integers $d$ and $N$.
Suppose that $f:\ZZ\to\RR$ satisfies $f(i)\ge0$ for all $i\in \ZZ$. 
Let $(\ell_m)$ and $(r_m)$ be two sequences of integers with $-N\le \ell_m, r_m \le N$ for all $m$.
Then
\[ \delta(f) = \liminf_{m\to\infty} \frac{\| f(\ZZ\cap[-md-\ell_m,md+r_m]) \|}{2md+\ell_m+r_m+1}. \]
\end{lemma}

\begin{proof}
For each sufficiently large integer $n$, let $m := \lceil (n+N)/d\rceil$ and $m' := \lfloor (n-N)/d \rfloor$.
Then
\[ \max\{m'd+\ell_m,m'd+r_m\} \le m'd+N \le n \le md-N\le \min\{ md+\ell_m, md+r_m\}.\]
This implies
\[ [-m'd-\ell_m,m'd+r_m] \subseteq [-n,n] \subseteq [-md-\ell_m,md+r_m].\]
Since $f(i)\ge0$ for all $i\in \ZZ$, we have
\begin{multline*}
\frac{2m'd+\ell_m+r_m+1}{2n+1} \cdot \frac{ \| f([-m'd-\ell_m,m'd+r_m]) \|}{2m'd+\ell_m+r_m+1} \le \frac{\| f(\ZZ\cap[-n,n]) \|}{2n+1} \\
\le \frac{ 2md+\ell_m+r_m+1}{2n+1} \cdot \frac{ \| f([-md-\ell_m,md+r_m] )\| }{2md+\ell_m+r_m+1}.
\end{multline*}
Since $n\to\infty$ implies $m\to\infty$ and $m'\to\infty$, and since
\[ \lim_{n\to\infty} \frac{2m'd+\ell_m+r_m+1}{2n+1} = \lim_{n\to\infty} \frac{ 2md+\ell_m+r_m+1}{2n+1} = 1,\]
taking limit inferior of the above bounds as $n\to\infty$ establishes the result.
\end{proof}

In Section~\ref{sec:intro} we defined a periodic set and its period.
The next lemma gives the density of a periodic set.

\begin{lemma}\label{lem:period}
Let $U$ be a periodic subset of $\ZZ$ with period $d$.
Then $\delta(U) = |U\cap[1,d]|/d$.
\end{lemma}

\begin{proof}
Since $U$ has period $d$, we have 
\[ 2m|U\cap[1,d]| \le |U\cap[-md,md]|\le 2m|U\cap[1,d]|+1.\]
Applying Lemma~\ref{lem:density} with $\ell_m=r_m=0$ gives the result.
\end{proof}

Now we study dominating sets of the digraph $\Gamma(\ZZ,S)$, where $S\subseteq\ZZ\setminus \{0\}$.
We extend a previous result~\cite[Theorem~4]{IndepRatio} on the domination ratio of an undirected integer distance graph to a directed integer distance graph.  The proof is similar, but requires some minor adjustments to deal with directed edges.

Assume $S$ is a finite subset of $\ZZ\setminus\{0\}$. Let 
\[ a := \max S\cup\{0\}, \quad b:=-\min S\cup\{0\},\quad\text{and}\quad c:=a+b.\]
Also let $[m,n]$ denote the set $\{x\in\ZZ: m\le x\le n\}$.
A \emph{state} is a subset of $[1,c]$.
For every state $T$, there exists a dominating set $D$ of $\Gamma(\ZZ,S)$ such that $D\cap[ic+1,(i+1)c]=T+ic$ for some $i\in\ZZ$.\footnote{
In a more general context~\cite{IndepRatio} the states do not all satisfy this condition, and those which do are called \emph{admissible}.}
For example, the set $D:=T\cup(\ZZ\setminus[1,c])$ satisfies $D\cap[1,c]=T$ and is a dominating set since every integer in $[1,c]=[1,a]\cup[a+1,a+b]$ is dominated by some element of $D$ by the following argument.
\begin{itemize}
\item
If $a=0$ then $[1,a]=\emptyset$. If $a>0$ then $a\in S$ and thus every integer in $[1,a]$ is dominated by some integer in $[1-a,0]\subseteq D$.
\item
If $b=0$ then $[a+1,a+b]=\emptyset$. If $b>0$ then $-b\in S$ and thus every integer in $[a+1,c]$ is dominated by some integer in $[c+1,c+b]\subseteq D$.
\end{itemize}

A \emph{transition} occurs from a state $T$ to another state $T'$ if there exists a dominating set $D$ of $\Gamma(\ZZ,S)$ such that $D\cap[ic+1,(i+1)c]=T+ic$ and $D\cap[(i+1)c+1,(i+2)c]=T'+(i+1)c$ for some $i\in\ZZ$.
We may choose $i=0$, without loss of generality.
Moreover, the definition of $a$ and $b$ implies that a transition occurs from $T$ to $T'$ if and only if every element of $[a+1,c+a]$ is either in $T\cup(T'+c)$ or dominated by $T\cup (T'+c)$, since such an element cannot be dominated by any integer outside $[1,2c]$.

The \emph{state graph} associated with $\Gamma(\ZZ,S)$ is a digraph whose vertices are the states and whose edges are transitions.
The \emph{weight} of a state $T$ is $|T|/c$.
A \emph{doubly infinite walk} in the state graph is a sequence $(T_i:i\in\ZZ)$ of states such that $(T_i,T_{i+1})$ is an edge for all $i\in\ZZ$.
The \emph{lower average weight} of this walk is 
\[ \liminf_{m\to\infty} \sum_{i\in[-m,m]} \frac{ |T_i| }{(2m+1)c}. \]

\begin{proposition}\label{prop:period}
Assume $S$ is a finite subset of $\ZZ\setminus\{0\}$.
Let $a:=\max S\cup\{0\}$, $b:=-\min S\cup\{0\}$, and $c:=a+b$.
Then the domination ratio of $\Gamma(\ZZ,S)$ is achieved by some periodic dominating set with period at most $c2^c$.
\end{proposition}

\begin{proof}
A doubly infinite walk $(T_i:i\in\ZZ)$ in the state graph of $\Gamma(\ZZ,S)$ corresponds to a set 
\[ D:=\bigcup_{i\in\ZZ} (T_i+ic).\]
We show that $D$ is a dominating set, i.e., every integer $j\notin D$ is dominated by $D$.

We have $j\in [ic+1,(i+1)c]$ for some $i\in\ZZ$.
We first assume $j\in[ic+1,ic+a]$.
Then $j$ cannot be dominated by any integer outside $[(i-1)c+1,(i+1)c]$ by the definition of $a,b,c$.
Since $(T_{i-1},T_i)$ is an edge in the state graph, there exists a dominating set $D'$ such that
\begin{align*}
D'\cap[(i-1)c+1,ic] &= T_{i-1}+(i-1)c,\\
D'\cap[ic+1,(i+1)c] &= T_i+ic.
\end{align*}
Thus $j$ must be dominated by $\left( T_{i-1}+(i-1)c\right) \cup \left( T_i+ic \right) = D\cap [(i-1)c+1,(i+1)c]$.

We next assume $j\in[ic+a+1,(i+1)c]$. 
Then $j$ cannot be dominated by any integer outside $[ic+1,(i+2)c]$ by the definition of $a,b,c$.
Since $(T_i, T_{i+1})$ is an edge in the transition graph, a similar argument as above implies that $j$ must be dominated by $D\cap[ic+1,(i+2)c]$.
Thus $D$ is a dominating set of $\Gamma(\ZZ,S)$.

Conversely, a dominating set $D$ of $\Gamma(\ZZ,S)$ corresponds to a doubly infinite walk $(T_i:i\in\ZZ)$ in the state graph, where $T_i:=D\cap[ic+1,(i+1)c]$.
The lower average weight of this walk equals
\[ \liminf_{m\to\infty} \sum_{i\in[-m,m]} \frac{|T_i|}{(2m+1)c} 
= \liminf_{m\to\infty} \frac{|D\cap [-mc+1,mc+c]|}{(2m+1)c} = \delta(D) \]
where the last equality follows from Lemma~\ref{lem:density}.

We know that the infimum of the lower average weights of doubly infinite walks is achieved by repeating some simple cycle in the state graph~\cite[Lemma~3]{IndepRatio}.
The length of this cycle is at most $2^c$, the total number of states.
Thus $\dgamma(\ZZ,S)$ is achieved by some periodic dominating set with period at most $c2^c$.
\end{proof}

By Proposition~\ref{prop:period}, the domination ratio of $\Gamma(\ZZ,S)$ is achieved by some periodic dominating set, whose period is denoted by $p$.
Let $\ZZ_p:=\{1,2,\ldots,p\}$ be the cyclic group of order $p$ under addition modulo $p$ and let $S_p$ be the subset of $\ZZ_p$ consisting of all least positive residues of elements in $S$ modulo $p$.
We conclude this section by giving a relation between the domination ratio $\dgamma(\ZZ,S)$ of the integer distance digraph $\Gamma(\ZZ,S)$ and the domination number $\gamma(\ZZ_p,S_p)$ of the circulant digraph $\Gamma(\ZZ_p,S_p)$.

\begin{proposition}\label{prop:finite}
Assume $S$ is a finite subset of $\ZZ\setminus\{0\}$.
Let $D$ be a dominating set of $\Gamma(\ZZ,S)$ with period $p$ such that $\dgamma(\ZZ,S) = \delta(D)= |D\cap[1,p]|/p$.
Then $D\cap[1,p]$ is a minimum dominating set of $\Gamma(\ZZ_p,S_p)$ and $\gamma(\ZZ_p,S_p) = |D\cap[1,p]| = \dgamma(\ZZ,S) p$.
\end{proposition}

\begin{proof}
We first show that $D\cap[1,p]$ is a dominating set of $\Gamma(\ZZ_p,S_p)$.
To see this, let $i\in[1,p]\setminus D$.
Then $i$ is dominated by $i-s\in D$ for some $s\in S$.
There exists $j\in D\cap[1,p]$ such that $i-s\equiv j\pmod p$, since $p$ is the period of $D$.
Then $i$ is dominated by $j$ in $\Gamma(\ZZ_p,S_p)$ since $i-j\equiv s\pmod p$. 
Thus $D\cap[1,p]$ is a dominating set of $\Gamma(\ZZ_p,S_p)$.

Now let $E$ be a minimum dominating set of $\Gamma(\ZZ_p,S_p)$.
We show that $\overline{E}:= \cup_{k\in\ZZ} (E+kp)$ is a dominating set of $\Gamma(\ZZ,S)$.
To see this, let $i\in\ZZ\setminus \overline{E}$. 
There exists $i'\in [1,p]$ such that $i\equiv i'\pmod p$.
Since $E$ is a dominating set of $\Gamma(\ZZ_p,S_p)$, there exists some $j'\in E$ such that $i'-j'\in S_p$.
This implies that $i-j\in S$ for some $j\equiv j'\pmod p$, i.e., $j$ dominates $i$. 
We also have $j\in \overline{E}$ by the definition of $\overline{E}$.
Thus $\overline{E}$ is a dominating set of $\Gamma(\ZZ,S)$.

Combining the above two paragraphs we have
\[ \dgamma(\ZZ,S)\le \delta(\overline{E}) = |E|/p = \gamma(\ZZ_p,S_p)/p \le |D\cap[1,p]|/p =\dgamma(\ZZ,S) \]
where the two inequalities must both be equalities.
The result follows immediately from this.
\end{proof}

\section{Basic results on domination ratio}\label{sec:basic}

In this section we prove some basic results on the domination ratio of an integer distance graph, as summarized in Proposition~\ref{prop:basic}.

\begin{lemma}\label{lem:subset}
Suppose $S\subseteq S'\subseteq \ZZ\setminus\{0\}$. 
Then $\dgamma(\ZZ,S) \ge \dgamma(\ZZ,S')$.
\end{lemma}

\begin{proof}
Since $S\subseteq S'\subseteq \ZZ\setminus\{0\}$, the Cayley graph $\Gamma(\ZZ,S)$ is a subgraph of $\Gamma(\ZZ,S')$.
Thus a dominating set of $\Gamma(\ZZ,S)$ is also a dominating set of $\Gamma(\ZZ,S')$.
The result follows immediately.
\end{proof}

\begin{proposition}\label{prop:neg}
For any $S\subseteq \ZZ\setminus\{0\}$ we have $\dgamma(\ZZ,S) = \dgamma(\ZZ,-S)$. 
\end{proposition}

\begin{proof}
The automorphism of the group $\ZZ$ defined by $i\mapsto -i$ for all $i\in\ZZ$ induces an isomorphism between the two digraphs $\Gamma(\ZZ,S)$ and $\Gamma(\ZZ,-S)$.
The result follows immediately.
\end{proof}

\begin{proposition}\label{prop:efficient}
Let $S$ be a finite subset of $\ZZ\setminus\{0\}$.
Then $\dgamma(\ZZ,S) \ge 1/(|S|+1)$ and the equality holds when there exists an efficient dominating set of $\Gamma(\ZZ,S)$.
\end{proposition}

\begin{proof}
Let $D$ be a dominating set of $\Gamma(\ZZ,S)$.
Since $S$ is finite, there exists $N>0$ such that $S\subseteq[-N,N]$.
Let $n$ be a positive integer larger than $N$.
Any $j\in\ZZ\cap[-n,n]\setminus D$ is dominated by some $i\in D$.
If $i<-n$ then $j<-n+N$; if $i>n$ then $j>n-N$.
Thus $D\cap[-n,n]$ dominates all but at most $2N$ elements of $[-n,n]$.
Each element of $D\cap[-n,n]$ can dominate at most $(|S|+1)$ elements.
It follows that 
\[ (|S|+1) |D\cap [-n,n]| \ge 2n+1-2N.\]
Hence
\[ \delta(D) = \liminf_{n\to\infty} \frac{|D\cap[-n,n]|}{2n+1} \ge \liminf_{n\to\infty} \frac{2n+1-2N}{(2n+1)(|S|+1)} = \frac1{|S|+1}.\]
Since $D$ is an arbitrary dominating set of $\Gamma(\ZZ,S)$, we have $\dgamma(\ZZ,S)\ge 1/(|S|+1)$.

Now suppose that there is an efficient dominating set $D$ of $\Gamma(\ZZ,S)$.
Similarly as above, counting all vertices dominated by $D\cap[-n,n]$ gives
\[ (|S|+1) |D\cap [-n,n]| \le 2n+1+2N. \]
This implies $\dgamma(\ZZ,S)\le \delta(D) \le 1/(|S|+1)$.
\end{proof}

We next show that, for any nonnegative integer $s$, there exists $S\subseteq\ZZ\setminus\{0\}$ with $|S|=s$ such that $\Gamma(\ZZ,S)$ admits an efficient dominating set and thus has domination ratio $\dgamma(\ZZ,S)=1/(s+1)$.

\begin{proposition}\label{prop:cong}
Suppose $S = \{ i_1(s+1)+1,i_2(s+1)+2,\ldots,i_s(s+1)+s\}$, where $s=|S|\ge0$ and $i_1,\ldots,i_s\in\ZZ$.
Then $\Gamma(\ZZ,S)$ has an efficient dominating set $(s+1)\ZZ$ and thus $\dgamma(\ZZ,S)=1/(s+1)$.
\end{proposition}

\begin{proof}
When $s=0$ the digraph $\Gamma(\ZZ,S)=\Gamma(\ZZ,\emptyset)$ admits an efficient dominating set $\ZZ$ and thus has domination ratio $1$.
Assume $s\ge1$ below.
Each $j\in\ZZ$ can be written as $j=k(s+1)+r$ for some $k\in\ZZ$ and $r\in[0,s]$.
If $r=0$ then $j\in (s+1)\ZZ$ and other elements of $(s+1)\ZZ$ cannot dominate $j$, since $S$ contains no element congruent to $0$ modulo $s+1$.
If $r\ne0$ then $j$ is dominated by $(k-i_r)(s+1)\in(s+1)\ZZ$ since 
\[ (k-i_r)(s+1)+(i_r(s+1)+r)=k(s+1)+r=j \]
and other elements of $(s+1)\ZZ$ cannot dominate $j$, since $S$ contains exactly one element congruent to $r$ modulo $s+1$.
Thus $(s+1)\ZZ$ is an efficient dominating set of $\Gamma(\ZZ,S)$.
It follows from Proposition~\ref{prop:efficient} that $\dgamma(\ZZ,S)=1/(s+1)$.
\end{proof}

\begin{proposition}\label{prop:divisor}
Let $d$ be a common divisor of all elements of $S$.
Then $\dgamma(\ZZ,S/d) = \dgamma(\ZZ,S)$.
\end{proposition}

\begin{proof}
We may assume $d>0$ by Proposition~\ref{prop:neg}.
For each integer $k\in[0,d-1]$, let $\Gamma_k$ be the subgraph of $\Gamma(\ZZ,S)$ whose vertex set is 
\[ [k]_d:=\{i\in\ZZ: i\equiv k\pmod d\}\]
and whose edge set consists of all ordered pairs $(i,j)$ with $i,j\in[k]_d$ and $j-i\in S$. 

Since $j-i\in S$ implies $i\equiv j \pmod d$ for any $i,j\in\ZZ$, there is no edge between $\Gamma_k$ and $\Gamma_\ell$ if $k\ne \ell$.
If $i\equiv j\pmod d$ then
\[ j-i\in S \Leftrightarrow \lfloor j/d \rfloor d - \lfloor i/d \rfloor d \in S \Leftrightarrow \lfloor j/d \rfloor - \lfloor i/d \rfloor \in S/d.\]
Thus for each integer $k\in[0,d-1]$, sending $i\in[k]_d$ to $\lfloor i/d \rfloor$ gives an isomorphism between $\Gamma_k$ and $\Gamma(\ZZ,S/d)$.

Therefore $\Gamma(\ZZ,S)$ is isomorphic to the union of $d$ copies of $\Gamma(\ZZ,S/d)$. 
The result then follows.
\end{proof}

\begin{corollary}\label{cor:01}
If $|S|\le1$ then $\dgamma(\ZZ,S)=1/(|S|+1)$.
If $|S|\ge1$ then $\dgamma(\ZZ,S)\le 1/2$.
\end{corollary}

\begin{proof}
This follows from Lemma~\ref{lem:subset}, Proposition~\ref{prop:cong}, and Proposition~\ref{prop:divisor}.
\end{proof}

Let $S$ be a nonempty subset of $\ZZ\setminus\{0\}$ and define $\pm S:=\{\pm s:s\in S\}$. 
Then $\Gamma(\ZZ,\pm S)$ can be viewed as an undirected graph.
There is a result~\cite[Lemma~19]{IndepRatio} similar to Proposition~\ref{prop:divisor} for the independence ratio of $\Gamma(\ZZ,\pm S)$.
Moreover, the previous results in this section imply the following results on the domination ratio of $\Gamma(\ZZ,\pm S)$.

\begin{proposition}\label{prop:undirected}
{\normalfont(i)} For any integer $s\ge1$ we have $\dgamma(\ZZ,\{\pm s\}) = 1/3$.

{\normalfont(ii)} For any nonempty set $S\subseteq \ZZ\setminus\{0\}$ we have $1/(|\pm S|+1) \le \dgamma(\ZZ,\pm S)\le 1/3$.

{\normalfont(iii)} If $s=5k\pm2$ for some $k\in\ZZ$ then $\dgamma(\ZZ,\{\pm1,\pm s\}) = 1/5$.
\end{proposition}

\begin{proof}
Proposition~\ref{prop:cong} and \ref{prop:divisor} imply (i).
Lemma~\ref{lem:subset} and Proposition~\ref{prop:efficient} imply (ii).
Proposition~\ref{prop:cong} implies (iii).
\end{proof}

\section{Lemmas about blocks}\label{sec:block}

Let $S$ be a finite nonempty subset of $\ZZ\setminus\{0\}$.
A dominating set $D$ of $\Gamma(\ZZ,S)$ can be written as $D=\{ x_i: i\in\ZZ \}$, where $x_i< x_{i+1}$ for all $i\in\ZZ$, and it partitions $\ZZ$ into a disjoint union of \emph{blocks} $B_i:=\{ x_i, x_i+1,\ldots, x_{i+1}-1\}$ for all $i\in \ZZ$.
The size of a block $B_i$ is $b_i:=|B_i| = x_{i+1}-x_i$ and we say that $B_i$ is a $b_i$-block.
We define the \emph{block structure} of a union of consecutive blocks to be the sequence of sizes of the blocks in this union. 
We identify a dominating set $D$ with the block structure of $\ZZ$, since $D$ is determined by the block structure of $\ZZ$ up to a translation by an integer. 
When $D$ has period $p$, we can write $D$ as a finite sequence $(b_1,\ldots,b_\ell)$ of positive integers $b_1,\ldots,b_\ell$ with $b_1+\ldots+b_\ell=p$, which is repeated infinitely in both directions.
For example, the block structure $(2\ 3)^5\ 7\ (3\ 4)^2$ corresponds to 15 consecutive blocks, first with ten blocks alternating between 2-blocks and 3-blocks, then a 7-block, then four blocks alternating between 3-blocks and 4-blocks.
Repeating this block structure infinitely in both directions gives $((2\ 3)^5\ 7\ (3\ 4)^2)^\infty$, which determines a periodic dominating set up to a translation.

Assume $S=\{1,s\}$ for some $s\in\ZZ\setminus \{0,1\}$ below.
We use block structures to construct dominating sets for $\Gamma(\ZZ,\{1,s\})$ and show upper bounds for the domination ratio $\dgamma(\ZZ,\{1,s\})$.

\begin{lemma}\label{lem:1s}
{\normalfont(i)} If $s=3k+2$ for some integer $k$, then $\dgamma(1,s) = 1/3$.

{\normalfont(ii)} If $s=3k+1$ or $s=-3k$ for some integer $k\ge1$, then $\dgamma(\ZZ,\{1,s\}) \le (k+1)/(3k+2)$.

{\normalfont(iii)} If $s=3k$ or $s=-3k+1$ for some integer $k\ge1$, then $\dgamma(\ZZ,\{1,s\}) \le 2k/(6k-1)$. 
\end{lemma}

\begin{proof}
(i) This is a special case of Proposition~\ref{prop:cong}, where we have a periodic dominating set determined by the block structure $3^\infty$.

(ii) One can check that $\Gamma(\ZZ,\{1,s\})$ has a periodic dominating set determined by $(3^k\ 2)^\infty$.
Since $D$ is periodic, it follows from Lemma~\ref{lem:period} that $\dgamma(\ZZ,\{1,s\}) \le \delta(D) = (k+1)/(3k+2)$.

(iii) Similarly to (ii), one can check that $\Gamma(\ZZ,\{1,s\})$ has a periodic dominating set determined by $(3^{k-1}\ 4\ 3^{k-1}\ 1)^\infty$.
By Lemma~\ref{lem:period}, $\dgamma(\ZZ,\{1,s\}) \le \delta(D) = 2k/(6k-1)$.
\end{proof}

We need the following lemma to establish the equalities in (ii) and (iii) of Lemma~\ref{lem:1s}.

\begin{lemma}\label{lem:block}
Let $D=\{x_i:i\in\ZZ\}$ be a dominating set of $\Gamma(\ZZ,\{1,s\})$. 
\begin{enumerate}
\item\label{item:size}
For each $i\in\ZZ$, the size of $B_i$ satisfies $1\le b_i\le s+1$ if $s>0$ or $1\le b_i\le -s+2$ if $s<0$.
\item\label{item:singleton}
If $i\in \ZZ$ and $b_i\ge3$, then $D$ contains $x_i-s+2, x_i-s+3,\ldots,x_i-s+b_i-1$.
\end{enumerate}
\end{lemma}

\begin{proof}
The result is trivial when $b_i\le 2$.
Assume $b_i\ge3$ below.
Then $D$ contains the elements $x_i-s+2, x_i-s+3,\ldots,x_i-s+b_i-1$ in order to dominate $x_i+2,x_i+3,\ldots,x_i+b_i-1$.
This implies $x_i-s+b_i-1\le x_i$ when $s>0$ or $x_i+b_i \le x_i-s+2$ when $s<0$.
The result follows.
\end{proof}

\section{Proof of Theorem~\ref{thm:1s}}\label{sec:cluster}

In this section we determine the domination ratio of the integer distance graph $\Gamma(\ZZ,\{1,s\})$ for any $s\in\ZZ\setminus\{0,1\}$.
The main idea is to merge blocks into coarser partitions of $\ZZ$.

\begin{lemma}\label{lem:toggle}
If $U\subseteq \ZZ$ and $f:\ZZ\to\RR$ satisfy all of the following conditions, then $\delta(U) = \delta(f)$.
\begin{enumerate}
\item
The set $\ZZ$ is the disjoint union of finite nonempty subsets $A_i$ for $i$ in some index set $I$.
\item\label{item:AiW}
There exists a constant $b$ such that $\max A_i-\min A_i\le b$ for all $i\in I$.
\item\label{item:AiU}
For each $i\in I$ we have $|U\cap A_i| = \|f(A_i)\|$. 
\item\label{item:f}
There exists a constant $N>0$ such that $0\le f(j)\le N$ for all $j\in\ZZ$.
\end{enumerate}
\end{lemma}

\begin{proof}
Let $n$ be a sufficiently large integer.
Let $X_n$ be the union of all $A_i$ contained in $[-n,n]$.
We have $|U\cap X_n| = \|f(X_n)\|$ by \eqref{item:AiU}.
If $j\in\ZZ\cap [-n,n]\setminus X_n$, then $j$ is contained in some $A_i\not\subseteq[-n,n]$ and thus $j\in[-n,-n+b-1]\cup[n-b+1,n]$ by \eqref{item:AiW}.
Combining this with \eqref{item:f} gives
\[ 0\le |\left(U\cap[-n,n]\right)\setminus X_n| \le |\left(\ZZ\cap[-n,n]\right)\setminus X_n| \le 2b,\]
\[ 0\le \| f\left(\ZZ\cap[-n,n]\right)\setminus X_n\| \le 2bN.\]
It follows that
\[ \frac{|U\cap [-n,n] |}{2n+1} - \frac{2b}{2n+1} \le \frac{ \| f(\ZZ\cap[-n,n]) \|}{2n+1} \le  \frac{|U\cap [-n,n] |}{2n+1} + \frac{2bN}{2n+1}.\]
Taking limit inferior of the above bounds as $n\to\infty$ gives the result.
\end{proof}

\begin{proposition}\label{prop:3k+1}
Let $s=3k+1$ or $s=-3k$ for some integer $k\ge1$. 
Then $\dgamma(1,s)=(k+1)/(3k+2)$.
\end{proposition}

\begin{proof}
By Lemma~\ref{lem:1s}, it suffices to show $\delta(D) \ge(k+1)/(3k+2)$ for an arbitrary dominating set $D$ of the digraph $\Gamma(\ZZ,\{1,s\})$.
We may assume $D=\{x_i:i\in\ZZ\}$, where $x_i<x_{i+1}$ for all $i\in \ZZ$.
The set $D$ partitions $\ZZ$ into a disjoint union of blocks $B_i:=\{x_i,x_i+1,\ldots,x_{i+1}-1\}$ for all $i\in \ZZ$.
Each block size $b_i:=|B_i|$ is at most $3k+2$ by Lemma~\ref{lem:block}~\eqref{item:size}.

We want to define a function $f:\ZZ\to\RR$ and partition $\ZZ$ into a disjoint union of finite nonempty subsets $A_i$ for all $i$ in some index set $I$ in such a way that Lemma~\ref{lem:toggle} applies and gives $\delta(D)=\delta(f)$.

\vskip3pt\noindent\textsf{Step 1}. 
We first define $f(x) := 0$ for all $x\in \ZZ\setminus D$ and initiate $I:=\emptyset$.

\vskip3pt\noindent\textsf{Step 2}. 
For each $i\in\ZZ$ with $b_i\le 3$ we define $f(x_i):=1$, insert $i$ into $I$, and let $A_i:=B_i$; we have 
\[ \max A_i - \min A_i = b_i-1 \le2 \quad\text{and}\quad \|f(A_i)\| = 1 = |A_i\cap D|.\]
Since the blocks are pairwise disjoint, at the end of this step we have a disjoint union of the sets $A_i$ for all $i\in I$, which equals the union of all blocks of size at most $3$.

\vskip3pt\noindent\textsf{Step 3}. 
For each $i\in \ZZ$ with $4\le b_i\le 3k+2$ we insert $i$ into $I$ and define
\[ f(x_i) := (b_i-1)/2\le (3k+1)/2.\]
By Lemma~\ref{lem:block}~\eqref{item:singleton}, there are $b_i-3$ consecutive 1-blocks 
\[ B_j=\{ x_i-s+2 \},\ B_{j+1}=\{x_i-s+3\},\ \ldots,\ B_{j+b_i-4}=\{x_i-s+b_i-2\}.\]
Let $A_i$ be the union of these 1-blocks together with $B_i$.
Lemma~\ref{lem:block}~\eqref{item:size} implies
\[ \max A_i - \min A_i = \begin{cases}
(x_i+b_i-1) - (x_i-s+2) \le 2s-2=6k, & \text{if } s=3k+1, \\
(x_i-s+b_i-2) - x_i \le -2s=6k, & \text{if } s=-3k.
\end{cases}\]
Delete $j+h$ from $I$ and redefine $f(x_{j+h}):=1/2$ for all $h=0,1,\ldots,b_i-4$. We have
\[ \|f(A_i)\| = (b_i-3)/2 + (b_i-1)/2 = b_i-2 = |A_i\cap D|.\]
Since we delete any $1$-block included in a set $A_i$ defined in this step, in the end we still have a disjoint union of $A_i$ for all $i\in I$, and this union equals $\ZZ$ as we include all the blocks.

Now for every $i\in\ZZ$ the nonempty set $A_i$ satisfies
\[ \max A_i - \min A_i \le 6k \quad\text{and}\quad \|f(A_i)\| = |A_i\cap D|. \]
For each $x\in\ZZ$ we have 
$ 0\le f(x)\le (3k+1)/2.$
Thus Lemma~\ref{lem:toggle} gives $\delta(D) = \delta(f)$.

It remains to show $\delta(f)\ge (k+1)/(3k+2)$.
For each integer $n>0$, let $x_r$ be the largest element of $D$ such that $B_r\subseteq(-\infty,n]$ if $s=3k+1$, or the smallest element of $D$ such that $B_r\subseteq[-n,\infty)$ if $s=-3k$.
We distinguish two cases below to define a cluster $C_r$, which is a union of certain blocks.

\vskip3pt\noindent\textsf{Case 1}.
Suppose $b_r\le 2$.
Define a cluster $C_r:=B_r$.
We have either $|C_r|=1$ and $1/2\le \|f(C_r)\|\le 1$, or $|C_r|=2$ and $\|f(C_r)\|=1$.

\vskip3pt\noindent\textsf{Case 2}.
Suppose $3\le b_r\le 3k+2$.
Then $D$ must contain $x_r-s+2, x_r-s+3, \ldots,x_r-s+b_r-1$ in order to dominate $x_r+2,x_r+3,\ldots,x_r+b_r-1$.
Define a cluster 
\[ C_r:=\begin{cases}
[x_r-s+2,x_r+b_r-1], & \textrm{if } s=3k+1, \\
[x_r,x_r-s+b_r-2], & \textrm{if } s=-3k.
\end{cases} \]
Then $C_r$ has size $3k+b_r-1$ and is the disjoint union of blocks $B_i$ for all $x_i\in D\cap C_r$.
Let $m_\ell$ be the number of blocks of size $\ell$ in $C_r$ for $1\le\ell\le 3k+2$.
We have 
\[ m_1+2m_2+ 3m_3 + \cdots+(3k+2)m_{3k+2} = 3k+b_r-1. \]
This implies 
$m_3+\cdots + m_{3k+2} \le k.$
It follows that 
\begin{eqnarray*}
\|f(C_r)\| &\ge & \frac{m_1}2 + m_2 + m_3 + \sum_{4\le \ell \le 3k+2} \frac{\ell-1}2 m_\ell \\
&=& \frac12 \sum_{1\le \ell\le 3k+2} \ell m_\ell - \frac12 \sum_{3\le \ell \le 3k+2} m_\ell \\
&\ge& \frac{3k+b_r-1}{2} - \frac{k}2 = \frac{2k+b_r-1}2.
\end{eqnarray*}

Now we have the cluster $C_r$.
When $s=3k+1$ we recursively write $(-\infty,c-1]$ as a disjoint union of clusters, where $c$ is the smallest element of the cluster $C_r$.
When $s=-3k$ we recursively write $[c+1,\infty)$ as a disjoint union of clusters, where $c$ is the largest element of the cluster $C_r$.
Let $Z_n$ be the union of all clusters contained in $[-n,n]$.
Let $n_j$ be the number of clusters of size $j$ in $Z_n$.
Then 
\[ \|f(Z_n)\| \ge \frac{n_1}2+n_2+\sum_{3\le \ell\le 3k+2} \frac{2k+\ell-1}2 n_{3k+\ell-1}, \]
\[ |Z_n| = n_1 + 2n_2 + \sum_{3\le \ell\le 3k+2} (3k+\ell-1) n_{3k+\ell-1}. \]
It is clear that $1/2 \ge (k+1)/(3k+2)$. 
Moreover, for $3\le \ell\le 3k+2$ one can check that 
\[ \frac{2k+\ell-1}{2(3k+\ell-1)} \ge \frac{k+1}{3k+2}.\]
Thus
\[ \|f(Z_n)\| / |Z_n| \ge (k+1)/(3k+2). \]
Since all clusters are intervals of size at most $6k+1$, we have $Z_n=\ZZ\cap[-n+\ell_n,n-r_n]$, where $\ell_n,r_n\in[0,6k]$. 
By Lemma~\ref{lem:density}, we have
\[ \delta(f) = \liminf_{n\to\infty} \frac{\|f(Z_n)\|}{|Z_n|} \ge \frac{k+1}{3k+2}.\]
This completes the proof.
\end{proof}

\begin{remark}
If $1\in S$ then $\Gamma(\ZZ,S)$ has a Hamiltonian (directed) path $\cdots\to-2\to-1\to0\to1\to2\to\cdots$.
This is the only possible Hamiltonian path in $\Gamma(\ZZ,S)$ if the elements of $S$ are all positive.
On the other hand, if $S$ contains $1$ and another integer $s\le -2$, then there exist other Hamiltonian paths in $\Gamma(\ZZ,S)$. 
For example, $\cdots\to2\to3\to0\to1\to-2\to-1\to\cdots$ is a Hamiltonian path in $\Gamma(\ZZ,\{1,-3\})$.
Thus for each integer $k\ge1$, the graphs $\Gamma(\ZZ,\{1,3k+1\})$ and $\Gamma(\ZZ,\{1,-3k\})$ are not isomorphic, even though they have the same domination ratio by Proposition~\ref{prop:3k+1}.
\end{remark}

\begin{proposition}\label{prop:3k}
If $s=3k$ or $s=-3k+1$ for some integer $k\ge1$, then $\dgamma(\ZZ,\{1,s\}) = 2k/(6k-1)$.
\end{proposition}

\begin{proof}
This result is proved in a similar way as Proposition~\ref{prop:3k+1}.
By Lemma~\ref{lem:1s}, it suffices to show $\delta(D) \ge 2k/(6k-1)$ for an arbitrary dominating set $D=\{x_i:i\in\ZZ\}$ of the digraph $\Gamma(\ZZ,\{1,s\})$.
The set $D$ partitions $\ZZ$ into a disjoint union of blocks $B_i$ for all $i\in \ZZ$ and each block size $b_i:=|B_i|$ is at most $3k+1$ by Lemma~\ref{lem:block}~\eqref{item:size}.
We want to define a function $f:\ZZ\to\RR$ and partition $\ZZ$ into a disjoint union of finite nonempty subsets $A_i$ for all $i$ in some index set $I$ in such a way that Lemma~\ref{lem:toggle} applies and gives $\delta(D)=\delta(f)$.

\vskip3pt\noindent\textsf{Step 1}. 
We first define $f(x):=0$ for all $x\in \ZZ\setminus D$ and initiate $I:=\emptyset$.

\vskip3pt\noindent\textsf{Step 2}. 
For each $i\in\ZZ$ with $b_i\le 3$ we define $f(x_i):=1$, insert $i$ into $I$, and set $A_i:=B_i$; we have 
\[ \max A_i - \min A_i = b_i-1 \le 2 \quad\text{and}\quad \|f(A_i)\| = 1 = |A_i\cap D|.\]

\vskip3pt\noindent\textsf{Step 3}. 
For each $i\in\ZZ$ with $4\le b_i\le 3k+1$ we insert $i$ into $I$ and define
\[ f(x_i) := \frac{b_i(3k-1) + 2-3k}{6k-1} \le \frac{(3k+1)(3k-1)+2-3k}{6k-1} = \frac{9k^2-3k+1}{6k-1}\le 3k/2.\]
By Lemma~\ref{lem:block}~\eqref{item:singleton}, there are $b_i-3$ consecutive 1-blocks 
\[ B_j:=\{x_i-s+2\}, B_{j+1}:= \{x_i-s+3\}, \ldots, B_{j+b_i-4} := \{x_i-s+b_i-2\}.\]
Let $A_i$ be the union of these 1-blocks together with $B_i$.
Then
\[ \max A_i - \min A_i = 
\begin{cases}
(x_i+b_i-1) - (x_i-s+2) \le 2s-2=6k-2 , & \text{if } s = 3k,\\
(x_i-s+b_i-2) - x_i \le -2s=6k-2, & \text{if } s = -3k+1.
\end{cases} \]
Delete $j+h$ from $I$ and redefine $f(x_{j+h}):=3k/(6k-1)$ for $h=0,1,\ldots,b_i-4$.
We have
\[ \|f(A_i)\| = \frac{(b_i-3)3k + b_i(3k-1) + 2-3k}{6k-1} = b_i-2 = |A_i\cap D|.\]

One sees that $\ZZ$ is the disjoint union of nonempty subsets $A_i$ with 
\[ \max A_i - \min A_i \le 6k-2 \quad\text{and}\quad \|f(A_i)\| = |A_i\cap D| \]
for all $i\in I$.
For each $x\in\ZZ$ we have 
$0\le f(x)\le 3k/2.$
Thus Lemma~\ref{lem:toggle} gives $\delta(D) = \delta(f)$.

It remains to show $\delta(f)\ge 2k/(6k+1)$.
For each integer $n>0$, let $x_r$ be the largest element of $D$ such that $B_r\subseteq(-\infty,n]$ if $s=3k$, or the smallest element of $D$ such that $B_r\subseteq[-n,\infty)$ if $s=-3k+1$.
We distinguish two cases below to define a cluster $C_r$.

\vskip3pt\noindent\textsf{Case 1}. 
Suppose $b_r\le 2$.
Define a cluster $C_r:=B_r$.
Then we have either $|C_r|=1$ and $3k/(6k-1)\le \|f(C_r)\|\le 1$, or $|C_r|=2$ and $\|f(C_r)\|=1$.

\vskip3pt\noindent\textsf{Case 2}. 
Suppose $3\le b_r\le 3k+1$.
Then $D$ contains $x_r-s+2,x_r-s+3,\ldots,x_r-s+b_r-1$ in order to dominate $x_r+2,x_r+3,\ldots,x+b_r-1$.
We define
\[ C'_r := \begin{cases}
[x_r-s+2,x_r+b_r-1], & \text{if } s=3k, \\
[x_r,x_r-s+b_r-2], & \text{if } s=-3k+1.
\end{cases}\]
Then $C'_r$ has size $3k+b_r-2$ and is the disjoint union of $B_i$ for all $i\in D\cap C'_r$.
Let $m_\ell$ be the number of blocks of size $\ell$ contained in $C'_r$ for $1\le \ell\le 3k+1$.
Then 
\[ m_1+2m_2+ 3m_3 + \cdots+(3k+1)m_{3k+1} = 3k+b_r-2. \]
This implies $m_3+\cdots + m_{3k+1} \le k.$
To define a cluster $C_r$ we further distinguish two subcases.

\vskip3pt\noindent\textsf{Case 2.1}. 
Suppose $m_1+m_2\ge1$.
Define a cluster $C_r:=C'_r$ which has size $3k+b_r-2$ and satisfies
\begin{eqnarray*}
\|f(C_r)\| &\ge& \frac{3k}{6k-1}m_1 + m_2 + m_3 + \sum_{4\le \ell\le 3k+1} \frac{\ell(3k-1)+2-3k}{6k-1}m_\ell \\
& = & \sum_{1\le\ell\le 3k+1} \frac{3k-1}{6k-1} \ell m_\ell + \frac{m_1+m_2}{6k-1} - \sum_{3\le\ell\le 3k+1} \frac{3k-2}{6k-1} m_\ell \\
& \ge & \frac{(3k-1)(3k+b_r-2)}{6k-1} + \frac{1}{6k-1} - \frac{(3k-2)k}{6k-1} \\
& = & \frac{6k^2+(3b_r-7)k-b_r+3}{6k-1}.
\end{eqnarray*}

\vskip3pt\noindent\textsf{Case 2.2}. 
Suppose $m_1+m_2=0$.
Then we have $b_r=3$ since $\{x_r-s+2\}$ is a 1-block when $b_r\ge4$. 
Define a cluster 
\[ C_r := \begin{cases}
[x_r-2s+4,x_r+2], & \text{if } s=3k,\\
[x_r,x_r-2s], & \text{if } s=-3k+1.
\end{cases}\]
Then $C_r$ has size $6k-1$ and is the disjoint union of $B_i$ for all $i\in D\cap C_r$ by the following argument.
\begin{itemize}
\item
If $s=3k$ then the block containing $x_r-s+2$ (the smallest element of $C'_r$) has size at least $3$ since $m_1+m_2=0$, and thus $D$ contains $x_r-2s+4$ in order to dominate $x_r-s+4$.
\item
If $s=-3k+1$ then the block containing $x_r-s+1$ (the largest element of $C'_r$) has size at least $3$ since $m_1+m_2=0$, and thus $D$ contains $x_r-2s+1$ to dominate $x_r-s+1$.
\end{itemize}
Let $t_\ell$ be the number of blocks of size $\ell$ contained in $C_r$ for $1\le \ell \le 3k+1$.
Then 
\[ t_1 + 2t_2 + 3t_3+\cdots+(3k+1)t_{3k+1} = 6k-1. \]
This implies $t_3+t_4+\cdots + t_{3k+1} \le 2k-1$.
Since $C'_r$ has size $3k+1$ and does not contain any block of size one or two, it must contain a block $B_j$ of size four or larger.
Then $D$ must contain $x_j-s+2$ and $x_j-s+3$ in order to dominate $x_j+2$ and $x_j+3$.
By the definition of $C_r$, we have $x_j-s+2\in C_r$ and thus $t_1\ge1$.
It follows that
\begin{eqnarray*}
\|f(C_r)\| &\ge & \frac{3k}{6k-1}t_1 + t_2 + t_3 + \sum_{4\le \ell\le 3k+1} \frac{\ell(3k-1)+2-3k}{6k-1}t_\ell \\
& = & \sum_{1\le\ell\le 3k+1} \frac{3k-1}{6k-1} \ell t_\ell + \frac{t_1+t_2}{6k-1} - \sum_{3\le\ell\le 3k+1} \frac{3k-2}{6k-1} t_\ell \\
& \ge & \frac{(3k-1)(6k-1)}{6k-1} + \frac{1}{6k-1} - \frac{(3k-2)(2k-1)}{6k-1} \\
&=& \frac{12k^2-2k}{6k-1}=2k.
\end{eqnarray*}

Now we have the cluster $C_r$.
When $s=3k$ we recursively write $(-\infty,c-1]$ as a disjoint union of clusters, where $c$ is the smallest element of $C_r$.
When $s=-3k+1$ we recursively write $[c+1,\infty)$ as a disjoint union of clusters, where $c$ is the largest element of $C_r$.
Let $Z_n$ be the union of all clusters contained in $[-n,n]$.
Let $n_j$ be the number of clusters of size $j$ in $Z_n$ defined in Case 1 and Case 2.1  for $j\in\{1,2\}\cup\{3k+\ell-2: 3\le \ell\le 3k+1\}$.
Let $n'_{6k-1}$ be the number of clusters in $Z_n$ defined in Case 2.2.
We have 
\[ \|f(Z_n)\| \ge \frac{3kn_1}{6k-1}+n_2+\sum_{3\le \ell\le 3k+1} \frac{6k^2+(3\ell-7)k-\ell+3}{6k-1} n_{3k+\ell-2}+ 2k n'_{6k-1}, \]
\[ |Z_n| =  n_1 + 2n_2 + \sum_{3\le \ell\le 3k+1} (3k+\ell-2) n_{3k+\ell-2} + (6k-1)n'_{6k-1}. \]
For $k\ge1$ and $3\le\ell\le 3k+1$ one can check that 
\[ \frac{3k}{6k-1} \ge \frac{2k}{6k-1}, \quad \frac12\ge \frac{2k}{6k-1},
\quad\text{and}\quad \frac{6k^2+(3\ell-7)k-\ell+3}{(6k-1)(3k+\ell-2)} \ge \frac{2k}{6k-1}. \]
Thus
\[ \|f(Z_n)\| / |Z_n| \ge 2k/(6k-1). \]
Since all clusters are intervals of size at most $6k-1$, we have $Z_n=\ZZ\cap[-n+\ell_n,n-r_n]$, where $\ell_n,r_n\in[0,6k-2]$. 
It follows from Lemma~\ref{lem:density} that
\[ \delta(f) = \liminf_{n\to\infty} \frac{\|f(Z_n)\|}{|Z_n|} \ge \frac{2k}{6k-1}.\]
This completes the proof.
\end{proof}

\begin{corollary}
Let $s$ and $t$ be two distinct elements of $\ZZ\setminus\{0,1\}$ such that $s\mid t$.
Then there exists an efficient dominating set for $\Gamma(\ZZ,\{s,t\})$ if and only if $t/s\equiv2\pmod 3$.
\end{corollary}

\begin{proof}
By the proof of Proposition~\ref{prop:divisor}, $\Gamma(\ZZ,\{s,t\})$ is the disjoint union of subgraphs $\Gamma_0,\ldots,\Gamma_{s-1}$, which are all isomorphic to $\Gamma(\ZZ,\{1,t/s\})$.
First assume $t/s \equiv2\pmod 3$.
Then $\Gamma(\ZZ,\{1,t/s \})$ has an efficient dominating set by the proof of Lemma~\ref{lem:1s} (i).
Thus $\Gamma(\ZZ,\{s,t\})$ has an efficient dominating set equal to the union of efficient dominating sets of $\Gamma_0,\ldots,\Gamma_{s-1}$.

Now assume $t/s \not\equiv 2\pmod 3$.
We have $\dgamma(\ZZ,\{s,t\}) = \dgamma(\ZZ,\{1,t/s\}) > 1/3$ by Proposition~\ref{prop:divisor}, Proposition~\ref{prop:3k+1}, and Proposition~\ref{prop:3k}.
Thus Proposition~\ref{prop:efficient} implies that there exists no efficient dominating set for $\Gamma(\ZZ,\{s,t\})$.
\end{proof}

\begin{corollary}\label{cor:finite}
Let $k$ be a positive integer.
We have $\gamma(\ZZ_{3k+2},\{\pm1\}) = \gamma(\ZZ_{3k+2},\{1,2\})=k+1$ and $\gamma(\ZZ_{6k-1},\{1,3k\}) =2k$.
\end{corollary}

\begin{proof}
This follows from Theorem~\ref{thm:1s}, Proposition~\ref{prop:finite}, and the proof of Lemma~\ref{lem:1s}.
\end{proof}

\section{Conclusion}

The domination number of a circulant (di)graph $\Gamma(\ZZ_n,S)$ with $S\subseteq \ZZ_n$ has been examined in various cases.
In this paper we initiate the study of the domination ratio $\gamma(\ZZ,S)$ of an integer distance (di)graph $\Gamma(\ZZ,S)$ with $S\subseteq \ZZ$, which is a natural infinite extension of $\Gamma(\ZZ_n,S)$.
This is also related the integer tiling problem as mentioned in the end of Section~\ref{sec:intro}.
We show that the domination ratio $\dgamma(\ZZ,S)$ can always be achieved by a periodic dominating set (Proposition~\ref{prop:periodic}), extending a similar result of Carraher, Galvin, Hartke, Radcliffe, and Stolee~\cite{IndepRatio} on the independence ratio.
We also provide some basic results on the domination ratio $\gamma(\ZZ,S)$ (Proposition~\ref{prop:basic}).
Our main result (Theorem~\ref{thm:1s}) gives the exact value of $\dgamma(\ZZ,S)$ when $S$ consists of two distinct nonzero integers $s$ and $t$ with $s\mid t$. 
Our proof of this new result is different from the discharging method used in earlier work~\cite{IndepRatio} on the independence ratio; in particular, we do not need the aforementioned periodicity.
Our result implies the domination number of certain circulant graphs (Corollary~\ref{cor:finite}), and also suggests that when $n$ is large and $s$ is close to $n/2$, the circulant graph $\Gamma(\ZZ_n,\{1,s\})$ should have its domination number close to the upper bound $\lceil n/3 \rceil$ given by Proposition~\ref{prop:eff} (ii).
To further extend our result, one could investigate the domination ratio $\gamma(\ZZ,S)$ at least in the following cases.
\begin{itemize}
\item 
The set $S$ consists of two nonzero integers $s$ and $t$ with $s\nmid t$.
\item 
The set $S$ consists of three nonzero integers $1$, $s$, and $t$. 
\item
The set $S$ consists of four nonzero integers and satisfies $S=-S$, so that $\Gamma(\ZZ,S)$ can be viewed as a $2$-regular undirected graph.
\end{itemize}

Finally, we ask a question on Theorem~\ref{thm:1s}. 
In our proof of this theorem, the construction of a dominating set to achieve the domination ratio $\dgamma(\ZZ,\{1,3k+1\})=(k+1)/(3k+2)$ (or $\dgamma(\ZZ,\{1,3k\}) =2k/(6k-1)$, resp.) and the argument to show that this ratio is minimum are the same as for $\dgamma(\ZZ,\{1,-3k\})=(k+1)/(3k+2)$ (or $\dgamma(\ZZ,\{1,-3k+1\})=2k/(6k-1)$, resp.).
We suspect that there is some more intuitive explanation for the equalities $\dgamma(\ZZ,\{1,3k+1\}) = \dgamma(\ZZ,\{1,-3k\})$ and $\dgamma(\ZZ,\{1,3k\})=\dgamma(\ZZ,\{1,-3k+1\})$.
The methods used to study integer tiling might be helpful for this question.

\section*{Acknowledgements}

The author is grateful to the anonymous referees for providing many valuable constructive suggestions on this paper and for pointing out the connection to the integer tiling problem.

\end{document}